\documentclass[10pt]{amsart}
\usepackage{amssymb}
\usepackage{mathrsfs}
\usepackage{amsmath}
\usepackage{amsfonts}
\usepackage{pifont}
\usepackage{graphicx,amssymb,mathrsfs,amsmath}
\usepackage{tocvsec2}
\usepackage{color}
\newtheorem{thm}{Theorem}[section]

\newtheorem{lem}[thm]{Lemma}
\newtheorem{prop}[thm]{Proposition}
\theoremstyle{definition}

\theoremstyle{remark}
\newtheorem{rem}[thm]{Remark}
\numberwithin{equation}{section}

\begin{document}
\title[Some new generalized inequalities for the ratio functions]{Generalized inequalities for ratio functions of trigonometric and hyperbolic functions}

\author{Marko Kosti\' c}
\address{Faculty of Technical Sciences,
University of Novi Sad,
Trg D. Obradovi\' ca 6, 21125 Novi Sad, Serbia}
\email{marco.s@verat.net}

\author{Yogesh J. Bagul}
\address{Department of Mathematics, K. K. M. College Manwath, Dist : Parbhani(M.S.) - 431505,
India}
\email{yjbagul@gmail.com}

\author{Christophe Chesneau}
\address{LMNO, University of Caen-Normandie, Caen, France}
\email{christophe.chesneau@unicaen.fr}

\begin{abstract}
The main aim of this note, which can be viewed as a certain addendum to the paper \cite{2019}, is to propose several generalized inequalities for the ratio functions of trigonometric and hyperbolic functions. We basically follow the approach obeyed in this paper.
\\[2mm] {\it AMS Mathematics Subject Classification $(2010)$}: 26D20, 26D07, 33B30.
\\[1mm] {\it Key words and phrases:} Ratio functions of trigonometric and hyperbolic functions, inequalities, infinite products.
\end{abstract}

\maketitle

\section{Introduction}\label{intro}

The reading of paper \cite{2019} by C. Chesneau and Y. J. Bagul
has strongly influenced us to write this note. In paper \cite{2019} the inequalities involving $ \cosh x/\cos x $ and $ \sinh x/\sin x $ were established by using refinement of Bernoulli inequality. We establish corresponding several generalized inequalities by using further refinement of Bernoulli type inequality. The refinement of Bernoulli type inequality can be of independent interest.

\section{Inequalities for the ratio functions of trigonometric and hyperbolic functions}\label{improve}
The following result presents a sharp upper bound for $\ln  [(1+uv)/(1-uv)]$ involving $\ln  [(1+v)/(1-v)]$ and polynomial terms in $u$ and $v$. 
\begin{lem}\label{izc}
Suppose $u,\ v\in (0,1)$ and $k_{0}\in \{-1,0\} \cup {\mathbb N}.$ 
Then we have
\begin{align*}
\ln \left(\frac{1+uv}{1-uv}\right) \leq 2\sum_{k=0}^{k_{0}}\frac{v^{2k+1}\bigl[ u^{2k+1}-u^{2k_{0}+3} \bigr]}{2k+1}+u^{2k_{0}+3}\ln \left(\frac{1+v}{1-v}\right).
\end{align*}
\end{lem}

\begin{proof}
Owing to the power series expansions of  $\ln  [(1+uv)/(1-uv)]$ and the fact that, for each natural number $k\geq k_{0}+1$, we have
$u^{2k_{0}+3}\geq u^{2k+1},$ we get
\begin{align*}
\ln \left(\frac{1+uv}{1-uv}\right)&=2\sum_{k=0}^{k_{0}}\frac{(uv)^{2k+1}}{2k+1}+2\sum_{k=k_{0}+1}^{+\infty}\frac{(uv)^{2k+1}}{2k+1}
\\& \leq 2\sum_{k=0}^{k_{0}}\frac{(uv)^{2k+1}}{2k+1}+2u^{2k_{0}+3}\sum_{k=k_{0}+1}^{+\infty}\frac{v^{2k+1}}{2k+1}
\\& =2\sum_{k=0}^{k_{0}}\frac{(uv)^{2k+1}}{2k+1}+u^{2k_{0}+3}\Biggl[ \ln \left(\frac{1+v}{1-v}\right) -\sum_{k=0}^{k_{0}}\frac{v^{2k+1}}{2k+1}\Biggr]
\\& =2\sum_{k=0}^{k_{0}}\frac{v^{2k+1}\bigl[ u^{2k+1}-u^{2k_{0}+3} \bigr]}{2k+1}+u^{2k_{0}+3}\ln \left(\frac{1+v}{1-v}\right).
\end{align*}
The proof of Lemma \ref{izc} is completed. 
\end{proof}
\begin{rem}
The inequality in Lemma \ref{izc} is equivalent to 
\begin{align*}
\frac{1+uv}{1-uv} \leq \exp\left\lbrace 2\sum_{k=0}^{k_{0}}\frac{v^{2k+1}\bigl[ u^{2k+1}-u^{2k_{0}+3} \bigr]}{2k+1}\right\rbrace \left(\frac{1+v}{1-v}\right)^{u^{2k_{0}+3}}.
\end{align*}
\end{rem}

It is worth noting that Lemma \ref{izc} further refine the Bernoulli type inequality established in \cite[Proposition 2]{2019}.

More to the point, we have the following:

\begin{prop}\label{profa}
Suppose $u,\ v\in (0,1).$ For any $k\in \{-1,0\} \cup {\mathbb N}$, let 
\begin{align}\label{ak}
a_{k}:=2\sum_{j=0}^{k}\frac{v^{2j+1}\bigl[ u^{2j+1}-u^{2k+3} \bigr]}{2j+1}+u^{2k+3}\ln \left(\frac{1+v}{1-v}\right),
\end{align}
such as, by Lemma \ref{izc}, we have $\ln  [(1+uv)/(1-uv)] \leq a_{k_0}$. Then, the sequence $(a_{k})_{k\in \{-1,0\} \cup {\mathbb N}}$ is strictly monotonically decreasing.
\end{prop}

\begin{proof}
Let $k\in \{-1,0\} \cup {\mathbb N}.$ We have 
\begin{align*}
& a_k-a_{k+1}\\
& = 2\frac{(uv)^{2k+3}}{2k+3}\bigl( u^{2}-1 \bigr)+2\bigl( u^{2}-1 \bigr)\sum_{j=0}^{k}\frac{v^{2j+1}u^{2k+3}}{2j+1}-u^{2k+3}\bigl( u^{2}-1 \bigr) \ln \left(\frac{1+v}{1-v}\right)\\
& = u^{2k+3}(u^2-1)\left[  2\frac{v^{2k+3}}{2k+3}+2\sum_{j=0}^{k}\frac{v^{2j+1}}{2j+1}-\ln \left(\frac{1+v}{1-v}\right) \right]\\
& = 2u^{2k+3}(1-u^2)\sum_{j=k+2}^{+\infty}\frac{v^{2j+1}}{2j+1}>0.
\end{align*}
Hence, $a_{k}> a_{k+1}$, implying the desired result. This ends the proof of Proposition \ref{profa}.
\end{proof}

Now we will prove the following extension of \cite[Proposition 2]{2019}, where the case $k_{0}=-1$ has been considered.

\begin{prop}\label{sorry}
For each number $k\in {\mathbb N}_{0},$ set 
$$
I_{k}:=\sum_{n=1}^{+\infty}\frac{1}{(2n-1)^{4k+2}}.
$$
Suppose $\alpha \in (0,\pi/2),$ $x\in (0,\alpha)$ and $k_{0}\in \{-1,0\} \cup {\mathbb N}.$  
Then, we have
\begin{align*}
\frac{\cosh x}{\cos x}
\leq  \Biggl(\frac{\cosh \alpha}{\cos \alpha} \Biggr)^{(x/\alpha)^{4k_{0}+6}}
 \exp\left\lbrace 2\sum_{k=0}^{k_{0}}\frac{(4\alpha^{2}/\pi^{2})^{2k+1}\bigl[ (x/\alpha)^{4k+2}-(x/\alpha)^{4k_{0}+6} \bigr]}{2k+1}I_{k}\right\rbrace.
\end{align*}
Denote by $b_{k_{0}}$ the right hand side of this inequality. Then, the sequence $(b_{k})_{k\in \{-1,0\} \cup {\mathbb N}}$ is monotonically decreasing.
\end{prop}
\begin{proof}
The proof follows the arguments of those of the proof of \cite[Proposition 2]{2019}. We consider the following product expansion:
$$
\frac{\cosh x}{\cos x}=\prod_{n=1}^{+\infty}\frac{1+ 4 x^{2}/[\pi^{2}(2n-1)^{2}]}{1-4x^{2}/[\pi^{2}(2n-1)^{2}]}.
$$
Applying Lemma \ref{izc} with $u=x^{2}/\alpha^{2}$ and $v=4\alpha^{2}/[\pi^{2}(2n-1)^{2}]$, we get that
\begin{align*}
& \frac{\cosh x}{\cos x} = \prod_{n=1}^{+\infty}\frac{1+(x^{2}/\alpha^{2}) \{4\alpha^{2}/[\pi^{2}(2n-1)^{2}]\}}{1-(x^{2}/\alpha^{2}) \{4\alpha^{2}/[\pi^{2}(2n-1)^{2}]\}}
\\& \leq \prod_{n=1}^{+\infty} \Biggl( \frac{1+4\alpha^{2}/[\pi^{2}(2n-1)^{2}]}{1-4\alpha^{2}/[\pi^{2}(2n-1)^{2}]} \Biggr)^{(x/\alpha)^{4k_{0}+6}}\times \\
&  \exp\left\lbrace 2\sum_{k=0}^{k_{0}}\frac{\{4\alpha^{2}/[\pi^{2}(2n-1)^{2}]\}^{2k+1}\bigl[ (x/\alpha)^{4k+2}-(x/\alpha)^{4k_{0}+6} \bigr]}{2k+1}\right\rbrace
\\& =\Biggl(\frac{\cosh \alpha}{\cos \alpha} \Biggr)^{(x/\alpha)^{4k_{0}+6}}\times \\
&  \exp\left\lbrace 2\sum_{k=0}^{k_{0}}\sum_{n=1}^{+\infty}\frac{1}{(2n-1)^{4k+2}}\frac{(4\alpha^{2}/\pi^{2})^{2k+1}\bigl[ (x/\alpha)^{4k+2}-(x/\alpha)^{4k_{0}+6} \bigr]}{2k+1}\right\rbrace
\\ & =\Biggl(\frac{\cosh \alpha}{\cos \alpha} \Biggr)^{(x/\alpha)^{4k_{0}+6}} \exp\left\lbrace 2\sum_{k=0}^{k_{0}}\frac{(4\alpha^{2}/\pi^{2})^{2k+1}\bigl[ (x/\alpha)^{4k+2}-(x/\alpha)^{4k_{0}+6} \bigr]}{2k+1}I_{k}\right\rbrace.
\end{align*}
Now, note that we can write  $b_k=\prod_{n=1}^{+\infty}\exp(a_{k,n})$, where $a_{k,n}$ is defined by \eqref{ak} with $u=x^{2}/\alpha^{2}$ and $v=4\alpha^{2}/[\pi^{2}(2n-1)^{2}]$. Since $(a_{k,n})_{k\in \{-1,0\} \cup {\mathbb N}}$ is monotonically decreasing by Proposition \ref{profa}, the same holds for 
$(b_{k})_{k\in \{-1,0\} \cup {\mathbb N}}$. 
This ends the proof of Proposition \ref{sorry}. 
\end{proof}

The following proposition is a consequence of Proposition \ref{sorry}. 

\begin{prop}\label{cons}
The upper bound in Proposition \ref{sorry} implies \cite[Proposition 3]{2019}, i.e., for $x\in (0,\pi/2)$, 
\begin{align*}
\frac{\cosh x}{\cos x}
\leq   \left( \frac{\pi^2+4x^2}{\pi^2-4x^2}\right)^{\pi^2/8}.
\end{align*}
\end{prop}
\begin{proof}
For any $a\in (0,1)$ and  $k_{0}\in \{-1,0\} \cup {\mathbb N},$ set
$$S_{k_0}(a)=\sum_{k=0}^{k_0} \frac{a^{2k+1}}{2k+1}.$$
Now, for any $k\in {\mathbb N}_{0},$ we have
$$I_k \le \sum_{n=1}^{+\infty}\frac{1}{(2n-1)^{2}}=\frac{\pi^2}{8}$$
and $(x/\alpha)^{4k+2}-(x/\alpha)^{4k_{0}+6} >0$ for $k=0,\ldots,k_0$. It follows from Proposition \ref{sorry} that
\begin{align*}
\frac{\cosh x}{\cos x} \le \Biggl(\frac{\cosh \alpha}{\cos \alpha} \Biggr)^{(x/\alpha)^{4k_{0}+6}} \exp\left\lbrace \frac{\pi^2}{4}\left[ S_{k_0}\left(\frac{4x^2}{\pi^2}\right)-\left(\frac{x}{\alpha}\right)^{4k_{0}+6} S_{k_0}\left(\frac{4\alpha^2}{\pi^2}\right)\right]\right\rbrace. 
\end{align*}
We end the proof of Proposition \ref{cons} by applying $k_0\rightarrow +\infty$. Indeed, we have $\lim_{k_0\rightarrow +\infty} (x/\alpha)^{4k_{0}+6}=0$ and, by using the power series expansions of  $\ln  [(1+u)/(1-u)]$, we get
\begin{align*}
& \lim_{k_0\rightarrow +\infty}S_{k_0}\left(\frac{4x^2}{\pi^2}\right)-\left(\frac{x}{\alpha}\right)^{4k_{0}+6} S_{k_0}\left(\frac{4\alpha^2}{\pi^2}\right)=\lim_{k_0\rightarrow +\infty}S_{k_0}\left(\frac{4x^2}{\pi^2}\right)\\
& =\frac{1}{2}\ln \left(\frac{1+4x^2/\pi^2}{1-4x^2/\pi^2}\right)=\frac{1}{2}\ln \left(\frac{\pi^2+4x^2}{\pi^2-4x^2}\right).
\end{align*}
The desired upper bound follows. 
\end{proof}

\begin{rem}
The term $I_k$ can be bounded sharply by well-known results on the  Riemann zeta function defined by $$\zeta(s)=\sum_{n=1}^{+\infty}\frac{1}{n^s}.$$
Indeed, after some algebraic manipulations, we get
\begin{align*}
 I_k = (1- 2^{-(4k+2)})\zeta(4k+2). 
\end{align*}
Thus, well-known upper bound for $\zeta(s)$ gives upper bound for $I_k$. For instance, it follows from \cite{batir} that 
$$\zeta(s)\le (1-2^{1-s})^{-1}.$$
Hence,
$$I_k\le  (1- 2^{-(4k+2)})  (1- 2^{-(4k+1)})^{-1}.$$
However, the benefit of such sharp inequality in our context need further developments that we leave for a future work. 
\end{rem}



\begin{rem}\label{thebest} 
If $\alpha \in (0,\pi/2),$ then the constant $\beta= \ln (\cosh \alpha /\cos \alpha)/\alpha^2$ is the best possible constant for which we have
$$
\frac{\cosh x}{\cos x}\leq e^{\beta x^{2}},\quad x\in (0,\alpha).
$$
In actual fact, the following estimate has been deduced in the proof of \cite[Proposition 2]{2019}:
$$
\frac{\cosh x}{\cos x}\leq \Biggl(\frac{\cosh \alpha}{\cos \alpha} \Biggr)^{(x/\alpha)^{2}}.
$$
This means that the function $f(x)=(\cosh x/\cos x)^{1/x^{2}},$ $x\in (0,\pi/2)$ is monotonically increasing. So, if $
\cosh x/\cos x \leq e^{\gamma x^{2}},$ $x\in (0,\alpha)$ for some real number $\gamma,$ then we must have
$ (\cosh x/\cos x)^{1/x^{2}}\leq e^{\gamma}.
$ Letting $x\rightarrow \alpha-,$ we get that $\gamma \geq \beta,$ as claimed.
\end{rem}

One can on the similar line prove the following extension of \cite[Proposition 4]{2019} where again the case $ k_{0} = -1 $ has been considered. \cite[Proposition 5]{2019} can also be obtained easily from the following proposition.

\begin{prop}\label{Yog}

Suppose $\alpha \in (0,\pi/2),$ $x\in (0,\alpha)$ and $k_{0}\in \{-1,0\} \cup {\mathbb N}.$  
Then, we have
\begin{align*}
\frac{\sinh x}{\sin x}
& \leq  \Biggl(\frac{\sinh \alpha}{\sin \alpha} \Biggr)^{(x/\alpha)^{4k_{0}+6}}
\times \\
&  \exp\left\lbrace 2\sum_{k=0}^{k_{0}}\frac{(\alpha^{2}/\pi^{2})^{2k+1}\bigl[ (x/\alpha)^{4k+2}-(x/\alpha)^{4k_{0}+6} \bigr]}{2k+1} \zeta(4k+2)\right\rbrace.
\end{align*}
\end{prop}


\begin{thebibliography}{90}

\bibitem{batir}
N. Batir, \emph{New inequalities for the Hurwitz zeta function}, Proc. Indian Acad. Sci. {\bf 118} (4) (2008), 495-503.
\bibitem{2019}
C. Chesneau and Y. J. Bagul,
\emph{Some new bounds for ratio functions of trigonometric and hyperbolic functions,}
Indian J. Math. 
{\bf 61} (2019), 153--160.

\bibitem{mitro}
D. S. Mitrinovi\' c,
\emph{Analytic Inequalities,} 
Springer, Berlin, 1970.

\end{thebibliography}
\end{document}